\documentclass[12pt]{amsart}
\usepackage{amscd}
\usepackage{amsmath}
\usepackage{amssymb}
\usepackage{units}
\usepackage[all]{xy}
\def\NZQ{\Bbb}               % the font for N,Z,Q,R,C

\def\ZZ{{\NZQ Z}}

\def\frk{\frak}               % font for "Fraktur"

\def\mm{{\frk m}}

\def\Phi{{\frk n}}
\def\Phi{{\frk N}}
\def\opn#1#2{\def#1{\operatorname{#2}}} % to make operators
\opn\chara{char} \opn\length{\ell} \opn\pd{pd} \opn\rk{rk}
\opn\projdim{proj\,dim} \opn\injdim{inj\,dim} \opn\rank{rank}
\opn\depth{depth} \opn\sdepth{sdepth} \opn\fdepth{fdepth}
\opn\grade{grade} \opn\height{height} \opn\embdim{emb\,dim}
\opn\codim{codim}  \opn\min{min} \opn\max{max}

\opn\Tr{Tr} \opn\bigrank{big\,rank}
\opn\superheight{superheight}\opn\lcm{lcm}
\opn\trdeg{tr\,deg}%\emph{
\opn\reg{reg} \opn\lreg{lreg} \opn\ini{in} \opn\lpd{lpd}
\opn\size{size}
\opn\div{div} \opn\Div{Div} \opn\cl{cl} \opn\Cl{Cl}
\opn\Spec{Spec} \opn\Supp{Supp} \opn\supp{supp} \opn\Sing{Sing}
\opn\Ass{Ass} \opn\Min{Min}
\opn\Ann{Ann} \opn\Rad{Rad} \opn\Soc{Soc}
\opn\Im{Im} \opn\Ker{Ker} \opn\Coker{Coker} \opn\Am{Am}
\opn\Hom{Hom} \opn\Tor{Tor} \opn\Ext{Ext} \opn\End{End}
\opn\Aut{Aut} \opn\id{id}  \opn\deg{deg}

\opn\nat{nat}
\opn\pff{pf}%   \pf exists already
\opn\Pf{Pf} \opn\GL{GL} \opn\SL{SL} \opn\mod{mod} \opn\ord{ord}
\opn\Gin{Gin} \opn\Hilb{Hilb}
\opn\aff{aff} \opn\con{conv} \opn\relint{relint} \opn\st{st}
\opn\lk{lk} \opn\cn{cn} \opn\core{core} \opn\vol{vol}
\opn\link{link} \opn\star{star}
\opn\gr{gr}

\def\pot#1#2{#1[\kern-0.28ex[#2]\kern-0.28ex]}

\opn\dirlim{\underrightarrow{\lim}}
\opn\inivlim{\underleftarrow{\lim}}
\let\to=\rightarrow

\def\Implies{\ifmmode\Longrightarrow \else
        \unskip${}\Longrightarrow{}$\ignorespaces\fi}
\def\implies{\ifmmode\Rightarrow \else
        \unskip${}\Rightarrow{}$\ignorespaces\fi}
\def\iff{\ifmmode\Longleftrightarrow \else
        \unskip${}\Longleftrightarrow{}$\ignorespaces\fi}

\let\:=\colon
\newtheorem{Theorem}{Theorem}[]
\newtheorem{Lemma}[Theorem]{Lemma}

\let\epsilon\varepsilon
\let\phi=\varphi
\let\kappa=\varkappa
\def\qed{\ifhmode\textqed\fi
      \ifmmode\ifinner\quad\qedsymbol\else\dispqed\fi\fi}
\def\textqed{\unskip\nobreak\penalty50
       \hskip2em\hbox{}\nobreak\hfil\qedsymbol
       \parfillskip=0pt \finalhyphendemerits=0}
\def\dispqed{\rlap{\qquad\qedsymbol}}

\opn\dis{dis}
\def\pnt{{\raise0.5mm\hbox{\large\bf.}}}

\opn\Lex{Lex}

\begin{document}
\title{  Construction of Neron Desingularization for Two Dimensional Rings}

\author{ Gerhard Pfister and Dorin Popescu }
\thanks{}

\address{Gerhard Pfister,  Department of Mathematics, University of Kaiserslautern, Erwin-Schr\"odinger-Str., 67663 Kaiserslautern, Germany}
\email{pfister@mathematik.uni-kl.de}

\address{Dorin Popescu, Simion Stoilow Institute of Mathematics of the Romanian Academy, Research unit 5,
University of Bucharest, P.O.Box 1-764, Bucharest 014700, Romania}
\email{dorin.popescu@imar.ro}

\begin{abstract} An algorithmic proof of the General Neron Desingularization theorem is given for $2$-dimensional local rings and morphisms with small singular locus.

 \noindent
  {\it Key words } : Smooth morphisms,  regular morphisms\\
 {\it 2010 Mathematics Subject Classification: Primary 13B40, Secondary 14B25,13H05,13J15.}
\end{abstract}

\maketitle

\vskip 0.5 cm

\section{Introduction}

The General Neron Desingularization Theorem, first proved by the second author has many important applications. One application is the generalization of Artin's famous approximation theorem  (Artin \cite{A}, Popescu \cite{P}, \cite{P1}).

\begin{Theorem} (General Neron Desingularization,  Andr\'e \cite{An}, Popescu \cite{P0}, \cite{P}, \cite{P1}, Swan \cite{S})\label{gnd}  Let $u:A\to A'$ be a  regular morphism of Noetherian rings and $B$ an  $A$-algebra of finite type. Then  any $A$-morphism $v:B\to A'$   factors through a smooth $A$-algebra $C$, that is $v$ is a composite $A$-morphism $B\to C\to A'$.
\end{Theorem}

The proof of this theorem is not constructive. Constructive proofs for one-dimensional  rings were given in A.\ Popescu, D.\ Popescu \cite{AP}, and Pfister, Popescu \cite{PP}. In this paper we will treat the $2$-dimensional case.

\section{Constructive General Neron Desingularization in a special case}

 Let $u:A\to A'$ be a flat morphism of Noetherian Cohen-Macaulay local rings of dimension $2$. Suppose that  the maximal ideal $\mm$ of $A$ generates the maximal ideal of $A'$ and the completions of $A,A'$ are isomorphic. Moreover suppose that $A'$ is Henselian,  and $u$ is a regular morphism.

  Let $B=A[Y]/I$, $Y=(Y_1,\ldots,Y_n)$. If $f=(f_1,\ldots,f_r)$, $r\leq n$ is a system of polynomials from $I$ then we can define the ideal $\Delta_f$ generated by all $r\times r$-minors of the Jacobian matrix $(\partial f_i/\partial Y_j)$.   After Elkik \cite{El} let $H_{B/A}$ be the radical of the ideal $\sum_f ((f):I)\Delta_fB$, where the sum is taken over all systems of polynomials $f$ from $I$ with $r\leq n$.
Then $B_P$, $P\in \Spec B$ is essentially smooth over $A$ if and only if $P\not \supset H_{B/A}$ by the Jacobian criterion for smoothness.
   Thus  $H_{B/A}$ measures the non smooth locus of $B$ over $A$.
  $B$ is {\em standard smooth} over $A$ if  there exists  $f$ in $I$ as above such that $B= ((f):I)\Delta_fB$.

  The aim of this paper is to give an easy algorithmic  proof of the following theorem.
\vskip 0.3 cm
 \begin{Theorem} \label{m} Any $A$-morphism $v:B\to A'$  such that $v(H_{B/A})A'$ is $\mm A'$-primary factors through a standard smooth $A$-algebra $B'$.
 \end{Theorem}
 
 \begin{proof}
 We choose   $\gamma, \gamma'\in v(H_{B/A})A'\cap A$ such that $\gamma, \gamma'$ is a regular sequence in $A$, let us say  $\gamma=\sum_{i=1}^qv(b_i)z_i$, $\gamma'=\sum_{i=1}^qv(b_i)z'_i$ for some 
\footnote{For the algorithm we have to choose $\gamma,\gamma'$ more carefully:  $\gamma\equiv \sum_{i=1}^qb_i(y')z_i$ modulo $(\gamma^t,\gamma'^{t})$,  $\gamma'\equiv \sum_{i=1}^qb_i(y')z'_i$ modulo $(\gamma^t,\gamma'^{t})$ with $z_i, z'_i\in A$, and $y'_i\equiv v(Y_i)$ modulo $\mm^N$ in $A$ , $N>>0$.}
 $b_i\in H_{B/A}$ and $z_i, z'_i\in A'$. Set $B_0=B[Z,Z']/(f,{\tilde f}) $, where $ f=-\gamma+\sum_{i=1}^qb_iZ_i\in B[Z]$, $Z=(Z_1,\ldots,Z_q)$,  ${\tilde f}=-\gamma'+\sum_{i=1}^qb_iZ'_i\in B[Z']$, $Z'=(Z'_1,\ldots,Z'_q)$
and let $v_0:B_0\to A'$ be the map of $B$-algebras given by $Z\to z$, $Z'\to z'$. 
 Changing $B$ by $B_0$ we may suppose that $\gamma, \gamma'\in H_{B/A}$.

We need the following  lemmas.

 \begin{Lemma} \label{l}
 \begin{enumerate}
 \item{}(\cite[Lemma 3.4]{P0})  Let $B_1$ be the symmetric algebra $S_B(I/I^2)$ of $I/I^2$ over\footnote{Let $M$ b e a finitely represented $B$-module and $B^m\xrightarrow{(a_{ij})} B^n\to M\to 0$ a presentation then $S_B(M)=B[T_1, \ldots, T_n]/J$ with $J=(\{\sum\limits^n_{i=1} a_{ij} T_i\}_ {j=1, \ldots, m})$.} $B$. Then $H_{B/A}B_1\subset H_{B_1/A}$ and  $(\Omega_{B_1/A})_{\gamma}$ is free over $(B_1)_{\gamma}$ for any $\gamma\in H_{B/A}$.

\item{} (\cite[Proposition 4.6]{S})   Suppose that  $(\Omega_{B/A})_{\gamma}$ is free over $B_{\gamma}$. Let $I'=(I,Y')\subset A[Y,Y']$, $Y'=(Y'_1,\ldots,Y'_n)$. Then $(I'/I'^2)_{\gamma}$ is free over $B_{\gamma}$.

\item{} (\cite[Corollary 5.10]{P1}) Suppose that $(I/I^2)_{\gamma}$ is free over  $B_{\gamma}$. Then a power of $\gamma$ is in  $ ((g):I)\Delta _g$ for some $g=(g_1,\ldots g_r)$, $r\leq n$ in $I$.
\end{enumerate}
\hfill\ \end{Lemma}

Using (1) of Lemma \ref{l}  we can reduce the proof   to the case when $\Omega_{B_{\gamma}/A}$ and  $\Omega_{B_{\gamma'}/A}$ are free over $ B_{\gamma}$ respectively $ B_{\gamma'}$.
Let $B_1$ be given by (1) of Lemma \ref{l}. The inclusion $B\subset B_1$ has a retraction $w$ which maps $I/I^2$ to zero. For the reduction we change $B,v$ by $B_1,vw$.

Using (2) from Lemma \ref{l} we may reduce the proof   to the case when  $(I/I^2)_{\gamma}$ (resp.  $(I/I^2)_{\gamma'}$) is free over $ B_{\gamma}$  (resp. $ B_{\gamma'}$). Indeed,
since $\Omega_{B_{\gamma}/A}$ is free over $ B_{\gamma}$ we see  that changing $I$ with $(I,Y')\subset A[Y,Y']$ we may suppose that $(I/I^2)_{\gamma}$ is free over $ B_{\gamma}$. Similarly, for $\gamma'$. 

Using (3) from Lemma \ref{l} we may reduce the proof   to the case when a power  of $\gamma$ (resp. $\gamma'$) is in  $ ((f):I)\Delta _f$ (resp.  $ ((f'):I)\Delta _{f'}$) for some $f=(f_1,\ldots f_r)$, $r\leq n$ and  $f'=(f'_1,\ldots f'_{r'})$, $r'\leq n$  from $I$.

We may now assume that a power $d$ (resp. $d'$) of $\gamma$ (resp. $\gamma'$) has the form $d\equiv P= \sum_{i=1}^qM_iL_i\ \mbox{modulo}\ I$,
$d'\equiv P'= \sum_{i=1}^{q'}M'_iL'_i\ \mbox{modulo}\ I$  
for some $r\times r$ (resp. $r'\times r'$)  minors $M_i$ (resp. $M'_i$)  
of $(\partial f/\partial Y) $ (resp. $(\partial f'/\partial Y) $) and $L_i\in ((f):I)$
(resp. $L'_i\in ((f'):I)$).

 The Jacobian matrix $(\partial f/\partial Y)$ (resp. $(\partial f/\partial Y)$) can be completed with $(n-r)$ (resp. $(n-r')$) rows from $A^n$ obtaining a square $n$ matrix $H_i$ (resp. $H'_i$) such that $\det H_i=M_i$ (resp.  $\det H'_i=M'_i$).
This is easy using just the integers $0,1$.

Let ${\bar A}=A/(d^3)$, ${\bar B}={\bar A}\otimes_AB$, ${\bar A}'=A'/(d^3A')$, ${\bar v}={\bar A}\otimes_Av$.

 We will now construct  a standard smooth $A$-algebra $D$ and an $A$-morphism $\omega:D\to A'$ such that $y=v(Y)\in \Im \omega +d^3A'$.

\begin{Lemma}\label{D} There exists a standard smooth $A$-algebra $D$ such that $\bar v$ factors through $\bar D=\bar A\otimes_A D$.
\end{Lemma}
\begin{proof}
Let $y'\in A^n$ be such that  $y=v(Y)\equiv y'$ modulo $(d^3,d'^3)A'$, let us say $y- y'\equiv\ d'^2\epsilon \ \mbox{modulo} \ d^3$ for $\epsilon\in d'A'^n$. Thus $I(y')\equiv 0$ modulo $(d^3,d'^3)A'$. 

Recall  that we have $d'\equiv P'$ modulo $I$ and so $P'(y')\equiv d'$ modulo $(d^3,d'^3)$ in $A$. Thus $P'(y')\equiv d's\ \mbox{modulo}\ d^3 $ for a certain $s\in A$ with $s\equiv 1$ modulo $d'$.

 Let $G'_i$ be the adjoint matrix of $H'_i$ and $G_i=L_iG'_i$. We have
$G_iH'_i=H'_iG_i=M'_iL'_i\mbox{Id}_n$
and so
$P'(y')\mbox{Id}_n=\sum_{i=1}^{q '}G_i(y')H'_i(y').$

But $H'_i$ is the matrix $(\partial f'_k/\partial Y_j)_{k\in [r'],j\in [n]}$ completed with some $(n-r')$ rows of $0,\ 1$. Especially we obtain
 \begin{equation}\label{identity}(\partial f'/\partial Y)G_i=
 M'_iL'_i(\mbox{Id}_{r'}|0).\end{equation}

 Then $t_i:=H'_i(y')\epsilon\in d' A'^n$
satisfies
$$G_i(y')t_i= M'_i(y')L'_i(y')\epsilon $$
 and so
$$\sum_{i=1}^q G_i(y')t_i=P'(y')\epsilon\equiv  d's\epsilon\ \mbox{modulo} \ d^3.$$
 It follows that
 $$s(y- y')\equiv d' \sum_{i=1}^{q'}G_i(y')t_i\ \mbox{modulo}\ d^3.$$
Note that $t_{ij}=t_{i1} $ for all $i\in [r']$ and $j\in [n]$ because the first $r'$ rows of $H'_i$ does not depend on $i$ (they are the rows of $(\partial f'/\partial Y)$). 
 
 Let
 \begin{equation}\label{def of h}h=s(Y-y')-d'\sum_{i=1}^{q'}G_i(y')T_i,\end{equation}
 where  $T_i=(T_1,\ldots,T_{r'}, T_{i,r'+1}\ldots, T_{i,n})$, $i\in [q']$ are new variables. We will use also $T_{ij}=T_i$ for $i\in [r']$, $j\in [n]$ because it is convenient  sometimes. The kernel of the map
$\bar{\phi}:{\bar A}[Y,T]\to {\bar A}'$ given by $Y\to y$, $T\to t$ contains $h$ modulo $d^3$. Since
$$s(Y-y')\equiv d'\sum_{i=1}^{q '}G_i(y')T_i\ \mbox{modulo}\ h$$
and
$$f'(Y)-f'(y')\equiv \sum_j(\partial f'/\partial Y_j)(y') (Y_j-y'_j)$$
modulo higher order terms in $Y_j-y'_j$, by Taylor's formula we see that for $p'=\max_i \deg f'_i$ we have
\begin{equation}\label{def of Q}s^{p'}f'(Y)-s^{p'}f'(y')\equiv  \sum_js^{p'-1}d'(\partial f'/\partial Y_j)(y') \sum_{i=1}^{q'}G_{ij}(y')T_{ij}+d'^2 Q\end{equation}
modulo $h$ where $Q\in T^2 A[T]^{r'}$.   We have $f'(y')\equiv d'^2b'\ \mbox{modulo}\ d^3$ for some $b'\in d'A^{r'}$. Then
\begin{equation}\label{def of g}g_i=s^{p'}b'_i+s^{p'}T_i+Q_i, \qquad i\in [r'] \end{equation}  modulo $d^3$ is in the kernel of $\bar\phi$. Indeed,  we have $s^{p'}f'_i=d'^2g_i\ \mbox{modulo}\ (h,d^3)$ because of (\ref{def of Q}). Thus
$d'^2{\bar\phi}(g)=d'^2g(t)\in (h(y,t),f'(y))\in d^3A'$ and so $g(t)\in d^3A'$, because $u$ is flat and $d'$ is regular on $A/(d^3)$. Set $E={\bar A}[Y,T]/(I,g,h)$ and let  ${\bar \psi}:E\to {\bar A'}$ be the map induced by ${\bar \phi}$. Clearly, ${\bar v}$ factors through $\bar \psi$ because ${\bar v}$ is the composed map ${\bar B}= {\bar A}[Y]/I\to E\xrightarrow{{\bar \psi}}{\bar A}'$.

Now we will see that there exist $s',s''\in E$ such that $E_{ss's''}$ is smooth over $\bar A$ and $\bar \psi$ factors through $E_{ss's''}$.

Note that the $r'\times r'$-minor  $s'$ of $(\partial g/\partial T)$ given by the first  $r'$-variables $T$ is from $s^{r'p'}+(T)\subset 1+(d',T)$ because $Q\in (T)^2$. Then $V=({\bar A}[Y,T]/(h,g))_{ss'}$ is smooth over $\bar A$. As in \cite{PP} we claim that $I{\bar A}[Y,T]\subset (h,g){\bar A}[Y,T]_{ss's''}$ for some $s''\in 1+(d',d^3,T)A[Y,T]$. Indeed, we have $P'I{\bar A}[Y,T]\subset (f')A[Y,T]\subset (h,g){\bar A}[Y,T]_s$ and so $P'(y'+s^{-1}d' G(y')T)I\subset (h,g,d^3)A[Y,T]_s$. Since  $P'(y'+s^{-1}d'G(y')T)\in P'(y')+d'(T)V$ we get $P(y'+s^{-1}d'G(y')T)\equiv d's''\ \mbox{modulo}\ d^3$ for some $s''\in 1+(T)A[Y,T]$. It follows that $s''I\subset (((h,g):d'),d^3)A[Y,T]_{ss'}$.  Thus $s''I$ is contained modulo $d^3$ in $ (0:_Vd')=0$ because $d'$ is regular on $V$, the map ${\bar A}\to V$ being flat. This shows our claim. It follows that
   $I\subset (d^3,h,g)A[Y,T]_{ss's''}$. Thus $E_{ss's''}\cong V_{s''} $ is a ${\bar B}$-algebra which is also standard smooth over $\bar A$.

 As $u(s)\equiv 1$ modulo $d'$ and ${\bar\psi}(s'),{\bar \psi}(s'')\equiv 1$ modulo $(d',d^3,t)$, $d,d',t\in \mm A'$ we see that $u(s),{\bar\psi}(s'), {\bar\psi}(s'')$ are invertible because  $A'$ is local. Thus $\bar\psi$ (and so $\bar v$) factors through the standard smooth $\bar A$-algebra $E_{ss's''}$, let us say by ${\bar\omega}:E_{ss's''}\to \bar A'$.
 
Now, let $Y'=(Y'_1,\ldots.Y'_n)$, and $D$ be the $A$-algebra isomorphic with  \\
 $(A[Y,T]/(I,h,g))_{ss's''}$ by $Y'\to Y$, $T\to T$. Since $A'$ is Henselian we may lift $\bar \omega$ to a map $(A[Y,T]/(I,h,g))_{ss's''}\to A'$ which will correspond to a map $\omega:D\to A'$. Then $\bar v$ factors through ${\bar D}$, let us say ${\bar B}\to {\bar D}\to {\bar A}'$, where the first map is given by $Y\to Y'$.  Note that $v$ does not factor through $D$. 
\hfill\ \end{proof}

 Let $\delta:B\otimes_AD\cong D[Y]/ID[Y]\to A'$ be the $A$-morphism given by $b\otimes \lambda\to v(b)\omega(\lambda)$.

Claim: $\delta$ factors through a special finite type $B\otimes_AD$-algebra $\tilde E$.

The proof will follow the proof of  Lemma \ref{D}.
Note that the map $\bar B\to \bar D$ is given by $Y\to Y'+d^3D$. Thus $I(Y')\equiv 0$ modulo $d^3D$. Set $\tilde y=\omega(Y')$. Since $\bar v$ factors through 
$\bar \omega$ we get
  $y-\tilde y=v(Y)-\tilde y\in d^3A'^n$, let us say $y-\tilde y=d^2\nu$ for $\nu\in dA'^n$.

Recall  that $P=\sum_iL_i\det H_i$ for  $L_i\in ((f):I)$. We have $d\equiv P$ modulo $I$ and so $P(Y')\equiv d$ modulo $d^3$ in $D$ because $I(Y')\equiv 0$ modulo $d^3D$. Thus $P(Y')=d{\tilde s}$ for a certain ${\tilde s}\in D$ with ${\tilde s}\equiv 1$ modulo $d$.
 Let ${\tilde G}'_i$ be the adjoint matrix of $H_i$ and ${\tilde G}_i=L_i{\tilde G}'_i$. We have
$\sum_i{\tilde G}_iH_i=\sum_i H_i{\tilde G}_i=P\mbox{Id}_n$
and so
$$d{\tilde s}\mbox{Id}_n=P(Y')\mbox{Id}_n=\sum_i{\tilde G}_i(Y')H_i(Y').$$

But $H_i$ is the matrix $(\partial f_i/\partial Y_j)_{i\in [r],j\in [n]}$ completed with some $(n-r)$ rows from $0,1$. Especially we obtain
 \begin{equation}\label{identity1}(\partial f/\partial Y)\sum_i{\tilde G}_i=(P\mbox{Id}_r|0).\end{equation}

 Then ${\tilde t}_i:=\omega(H_i(Y'))\nu\in dA'^n$
satisfies
$$\sum_i{\tilde G}_i(Y'){\tilde t}_i=P(Y')\nu=d{\tilde s}\nu$$
 and so
 $${\tilde s}(y-\tilde y)=d\sum_i\omega({\tilde G}_i(Y')){\tilde t}_i.$$
 Let
 \begin{equation}\label{def of h1}{\tilde h}={\tilde s}(Y-Y')-d\sum_i{\tilde G}_i(Y'){\tilde T}_i,\end{equation}
 where  ${\tilde T}=({\tilde T}_1,\ldots,{\tilde T}_n)$ are new variables. The kernel of the map
${\tilde\phi}:D[Y,{\tilde T}]\to A'$ given by $Y\to y$, ${\tilde T}\to {\tilde t}$ contains ${\tilde h}$. Since
$${\tilde s}(Y-Y')\equiv d\sum_i{\tilde G}_i(Y'){\tilde T}_i\ \mbox{modulo}\ {\tilde h}$$
and
$$f(Y)-f(Y')\equiv \sum_j(\partial f/\partial Y_j)((Y') (Y_j-Y'_j)$$
modulo higher order terms in $Y_j-Y'_j$, by Taylor's formula we see that for $p=\max_i \deg f_i$ we have
\begin{equation}\label{def of Q1}{\tilde s}^pf(Y)-{\tilde s}^pf(Y')\equiv  \sum_j{\tilde s}^{p-1}d(\partial f/\partial Y_j)(Y')\sum_i {\tilde G}_{ij}(Y'){\tilde T}_{ij}+d^2{\tilde Q}\end{equation}
modulo $\tilde h$ where ${\tilde Q}\in {\tilde T}^2 D[{\tilde T}]^r$.   We have $f(Y')=d^2{\tilde b}$ for some ${\tilde b}\in dD^r$. Then
\begin{equation}\label{def of g1}{\tilde g}_i={\tilde s}^p{\tilde b}_i+{\tilde s}^p{\tilde T}_i+{\tilde Q}_i, \qquad i\in [r] \end{equation}  is in the kernel of $\tilde\phi$. Indeed,  we have ${\tilde s}^pf_i=d^2{\tilde g}_i\ \mbox{modulo}\ {\tilde h}$ because of (\ref{def of Q1}) and $P(Y')=d{\tilde s}$. Thus
$d^2\phi({\tilde g})=d^2{\tilde g}(t)\in ({\tilde h}(y,{\tilde t}),f(y))=(0)$ and so ${\tilde g}({\tilde t})=0$. Set ${\tilde E}=D[Y,{\tilde T}]/(I,{\tilde g},{\tilde h})$ and let  ${\tilde\psi}:{\tilde E}\to A'$ be the map induced by $\tilde \phi$. Clearly, $v$ factors through $\tilde \psi$ because $v$ is the composed map $B\to B\otimes_AD\cong D[Y]/I\to {\tilde E}\xrightarrow{{\tilde\psi}} A'$.

Finally we will prove that there exist ${\tilde s}',{\tilde s}''\in{\tilde E}$ such that ${\tilde E}_{{\tilde s}{\tilde s}'{\tilde s}''}$ is standard smooth over $A$ and $\tilde\psi$ factors through ${\tilde E}_{{\tilde s}{\tilde s}'{\tilde s}''}$.

Note that the $r\times r$-minor  ${\tilde s}'$ of $(\partial {\tilde g}/\partial {\tilde T})$ given by the first  $r$-variables ${\tilde T}$ is from ${\tilde s}^{rp}+({\tilde T})\subset 1+(d,{\tilde T})$ because ${\tilde Q}\in ({\tilde T})^2$. Then ${\tilde V}=(D[Y,{\tilde T}]/({\tilde h},{\tilde g}))_{{\tilde s}{\tilde s}'}$ is smooth over $D$. We claim that $I\subset ({\tilde h},{\tilde g})D[Y,{\tilde T}]_{{\tilde s}{\tilde s}'{\tilde s}''}$ for some other ${\tilde s}''\in 1+(d,{\tilde T})D[Y,{\tilde T}]$. Indeed, we have $PID[Y]\subset (f)D[Y]\subset ({\tilde h},{\tilde g})D[Y,{\tilde T}]_{\tilde s}$ and so $P(Y'+{\tilde s}^{-1}d\sum_i{\tilde G}_i(Y'){\tilde T}_i)I\subset ({\tilde h},{\tilde g})D[Y,{\tilde T}]_{\tilde s}$. Since  $P(Y'+{\tilde s}^{-1}d\sum_i{\tilde G}_i(Y'){\tilde T}_i)\in P(Y')+d({\tilde T})$ we get $P(Y'+{\tilde s}^{-1}d\sum_i{\tilde G}_i(Y'){\tilde T}_i)=d{\tilde s}''$ for some ${\tilde s}''\in 1+({\tilde T})D[Y,{\tilde T}]$. It follows that ${\tilde s}''I\subset (({\tilde h},{\tilde g}):d)D[Y,{\tilde T}]_{{\tilde s}{\tilde s}'}$. Thus ${\tilde s}''I\subset (0:_{\tilde V}d)=0$, which shows our claim. It follows that
   $I\subset ({\tilde h},{\tilde g})D[Y,{\tilde T}]_{{\tilde s}{\tilde s}'{\tilde s}''}$. Thus ${\tilde E}_{{\tilde s}{\tilde s}'{\tilde s}''}\cong {\tilde V}_{{\tilde s}''} $ is a $B$-algebra which is also standard smooth over $D$ and $A$.

 As $\omega({\tilde s})\equiv 1$ modulo $d$ and ${\tilde\psi}({\tilde s}'),{\tilde \psi}({\tilde s}'')\equiv 1$ modulo $(d,{\tilde t})$, $d,{\tilde t}\in \mm A'$ we see that $\omega({\tilde s}),{\tilde\psi}({\tilde s}'), {\tilde\psi}({\tilde s}'')$ are invertible because  $A'$ is local. Thus ${\tilde\psi}$ (and so $v$) factors through the standard smooth $A$-algebra $B'={\tilde E}_{{\tilde s}{\tilde s}'{\tilde s}''}$.
\hfill\ \end{proof}

\section{The Algorithm} 

We obtain the following algorithm (which will be implemented in SINGULAR as a library and available 2017).

{\bf Algorithm Neron Desingularization}

Input: $N\in \ZZ_{>0}$ a bound\\
$A:=k[x]_{( x)}/J, J=( h_1, \ldots, h_p) \subseteq k[x],  x=(x_1, \ldots, x_t), k$ , a field\\
$B:=A[Y]/I, I=( g_1, \ldots, g_l)\subseteq k[x, Y], Y=(Y_1, \ldots, Y_n)$\\
$v:B\to A'\subseteq K[[x]]/JK[[x]]$ an $A$-morphism, given by $y'=( y'_1, \ldots, y'_n)\in k[ x]^n$, approximations $\mod( x)^N$ of $v(Y)$.

Output:  A Neron desingularization of $v:B\to A'$ or the message ''the bound is too small''

1. Compute $H_{B/A}=(b_1,\ldots,b_q)_B$ and $H_{B/A}\cap A$.

2.  if $\dim A/H_{B/A}\cap A=0$ choose $\gamma,\gamma'\in H_{B/A}\cap A$, a regular sequence in $A$ and go to 6.

3.   choose $\gamma,\gamma'\in H_{B/A}(y')$, a regular sequence in $A$.

 4. Write $\gamma\equiv \sum_{i=1}^qb_i(y')y'_{i+n}$ modulo $(\gamma^t,\gamma'^{t})$,  $\gamma'\equiv \sum_{i=1}^qb_i(y')y'_{i+n+q}$ modulo $(\gamma^t,\gamma'^{t})$  for some $t$ and  $y_j\in k[x]$.

 5. $g_{l+1}:=-\gamma+\sum_{i=1}^qb_iY_{i+n}$, $g_{l+2}:=-\gamma'+\sum_{i=1}^qb_iY_{i+n+q}$,
$Y=(Y_1,\ldots,Y_{n+2q})$, $y'=(y'_1,\ldots,y'_{n+2q})$, $I=(g_1,\ldots,g_{l+2})$; $l:=l+2$; $n:=n+2q$; $B=A[Y]/I$.

6. $B:=S_B(I/I^2)$, $v$ trivially extended.
Write $B:=A[Y]/I$, $n:=|Y|$, $Y:=Y,Z$, $Z=(Z_1,\ldots,Z_n)$, $I:=(I,Z)$, $B:=A[Y]/I$, $v$ trivially extended.

7. Compute $f=(f_1,\ldots,f_r)$, and $f'=(f'_1,\ldots,f'_{r'})$ such that
a power $d$ of $\gamma$, resp. $d'$ of $\gamma'$ is in $((f):I)\Delta_f$, 
resp. in $((f'):I)\Delta_f'$. 

if $(d^3,d'^3)\not \supseteq (x)^N$ return to "the bound is too small". 

8. Choose $r$-minors $M_i$  (resp. $r'$-minors $M'_i$) of $(\partial f/\partial Y)$, (resp. $(\partial f'/\partial Y)$) and $L_i\in ((f):I)$, (resp. $L'_i\in ((f'):I)$) such that for $P=\sum_i M_iL_i$ (resp. $P'=\sum_i M'_iL'_i$), $d\equiv P$ modulo $I$ (resp. $d'\equiv P'$ modulo $I$).

9. Complete the Jacobian matrix $(\partial f/\partial Y)$  (resp. $(\partial f'/\partial Y)$) by $(n-r)$ (resp. $(n-r'$) rows of $0,1$  
to obtain square matrices $H_i$ (resp. $H'_i$) such that $\det H_i=M_i$ (resp. $\det H'_i=M'_i$).

10. Write $P'(y')=d's$ modulo $d^3$ for $s\in A$, $s\equiv 1$ modulo $d'$.

11. For $i=1$ to $q'$ compute $G'_i$ the adjoint matrix of $H'_i$ and $G_i=L_iG'_i$.

12. $h:=s(Y-y')-d'\sum_iG_i(y')T_i$, $T_i=(T_1,\ldots, T_{r'},T_{i,r'+1},\ldots,T_{i,n})$.

13. $p':=\max_i\{\deg f'_i\}$
write 
$s^{p'}f'(Y)-s^{p'}f'(y')=\sum_j s^{p'-1}d'\partial f'/\partial Y(y')\sum_iG_{ij}(y')T_{ij}+d'^2Q$ modulo $h$ and $f'(y')=d'^2b'$ modulo $d^3$.
For $i=1$ to $r'$ $g_i:=s^{p'} b'_i+s^{p'}T_i+Q_i$.
14. Compute $s'$ the $r'$-minor of $(\partial g/\partial T)$ given by the first $r'$ variables and $s''$ such that 
$P(y'+s^{-1}d'\sum_i G_i(y')T)=d's''$ modulo $d^3$.

15. $D:=(A[Y',T]/(I,g,h))_{ss's''}$, $Y'=(Y'_1,\ldots,Y'_n)$, $g:=g(Y')$,
$I:=I(Y')$, $h:=h(Y')$.

Write $P(Y')=d\tilde s$, $\tilde s\equiv 1 $ modulo $d$.

16. Compute $\tilde G'_i$ the adjoint matrix of $H_i$ and $\tilde G_i=L_i\tilde G'_i$.

17. $\tilde h:={\tilde s}(Y-Y')-d\sum_{i=1}^q{\tilde G}_i{\tilde T}_i$, ${\tilde T}_i=({\tilde T}_1,\ldots,{\tilde T}_r,{\tilde T}_{i,r+1},\ldots,{\tilde T}_{i,n})$.

18.  $p:=\max_i\{\deg f_i\}$

 Write 

${\tilde s}^{p}f(Y)-{\tilde s}^{p}f(Y')=\sum_j {\tilde s}^{p-1}d\partial f/\partial Y(Y')\sum_i{\tilde G}_{ij}(Y'){\tilde T}_{ij}+d'^2{\tilde Q}$ modulo $\tilde h$ and $f(Y')=d^2{\tilde b}$,  ${\tilde b}\in dD^r$.

19. For $i=1$ to $r$,  ${\tilde g}_i:={\tilde s}^p{\tilde b}_i+\tilde s^p{\tilde T}_i+{\tilde Q}_i$.

20. Compute ${\tilde s}'$ the $r\times r$-minors of $(\partial {\tilde g}/\partial {\tilde T})$ given by the first $r$ variables of $\tilde T$.
Compute ${\tilde s}''$ such that 
$P(Y'+{\tilde s}^{-1}d\sum_i {\tilde G}_i(Y'){\tilde T})=d{\tilde s}''$.

21. return $D[Y,{\tilde T}/(I, {\tilde g},{\tilde h})_{{\tilde s}{\tilde s}'{\tilde s}''}$.

\end{document}